\documentclass[a4paper,10pt]{amsart}

\usepackage{amssymb}
\usepackage{bbm}
\usepackage{mathrsfs}
\usepackage[T5,T1]{fontenc}	

\normalsize

\newtheoremstyle{thm}{1ex}{2ex}{\itshape}{}{\bfseries}{}{0.9em}{\thmnumber{(#2)}\thmname{ #1}\thmnote{ (#3)}}
\newtheoremstyle{example}{1ex}{2ex}{\rm}{}{\bfseries}{}{0.8em}{\thmnumber{(#2)}\thmname{ #1}}

\theoremstyle{thm}
\newtheorem{lemma}{Lemma}[section]
\newtheorem{prop}[lemma]{Proposition}
\newtheorem{cor}[lemma]{Corollary}

\theoremstyle{example}
\newtheorem{exas}[lemma]{Examples}

\DeclareMathOperator{\nil}{Nil}
\DeclareMathOperator{\degsupp}{degsupp}
\DeclareMathOperator{\pic}{Pic}
\DeclareMathOperator{\spec}{Spec}
\newcommand{\dfgl}{\mathrel{\mathop:}=}

\begin{document}

\title{Torsion functors with monomial support}
\author{Fred Rohrer}
\address{{\fontencoding{T5}\selectfont Institute of Mathematics, Vietnam Academy of Science and Technology, 18 Ho\`ang Qu\'\ocircumflex{}c Vi\d\ecircumflex{}t, 10307 H\`a \nolinebreak N\d\ocircumflex{}i, Vi\d\ecircumflex{}t Nam}}
\email{fredrohrer0@gmail.com}
\thanks{The author was supported by the Swiss National Science Foundation.}
\subjclass[2010]{Primary 13D45; Secondary 14M25}
\keywords{Torsion functor, local cohomology, monomial ideal, \v{C}ech cohomology, flatness, Cox scheme}

\begin{abstract}
The dependence of torsion functors on their supporting ideals is investigated, especially in the case of monomial ideals of certain subrings of polynomial algebras over not necessarily Noetherian rings. As an application it is shown how flatness of quasicoherent sheaves on toric schemes is related to graded local cohomology.
\end{abstract}

\maketitle


\section*{Introduction}

Over a Noetherian ring (where rings are always understood to be commutative), torsion functors (and hence local cohomology functors) depend only on the radical of their supporting ideals. Without Noetherianness this need not hold. Supposing the supporting ideals to be monomial ideals of finite type of a polynomial algebra one might hope for it to still hold -- after all, monomial ideals behave quite independently of the base ring and hence might be unaffected by its potential lack of some nice properties. We will see that this is indeed true. However, the torsion functors with respect to a monomial ideal of finite type and its radical may be different, unless the nilradical of the base ring is nilpotent. But using a monomial variant of the notion of radical it is possible to get a satisfying result that is independent of the base ring and moreover can be generalised to a class of subalgebras of polynomial algebras.

There were two main reasons for writing this note. First, the recent preprint \cite{bc} by Botbol and Chardin on Castelnuovo-Mumford regularity over not necessarily Noetherian polynomial algebras contains the claim that local cohomology with respect to monomial ideals of finite type remains the same when we replace the supporting ideal by its radical. This aroused suspicions and turned out to be indeed wrong. Luckily, the aforementioned monomial variant of radical can save the day, and the moral might be that working over not necessarily Noetherian rings requires some meticulousness even in a monomial setting.

The second reason comes from the study of toric schemes, i.e., ``toric varieties over arbitrary base rings''. An overview of this continuing project including its motivation can be found in \cite{ts0}. With this approach to toric geometry one gets a relation between flatness of quasicoherent sheaves on a toric scheme and graded \v{C}ech cohomology over a corresponding restricted Cox ring with respect to a certain sequence of monomials. The ideal $J$ generated by these monomials is not necessarily equal to the restricted irrelevant ideal $I$, and thus -- in case graded \v{C}ech cohomology coincides with graded local cohomology (which it need not do in general) -- the question arises whether or not graded local cohomology distinguishes between $J$ and $I$. By applying our general results we will see that it does not, and thus we can relate flatness of quasicoherent sheaves on toric schemes with graded local cohomology with respect to $I$. (The importance of graded local cohomology with respect to $I$ stems from its relation to sheaf cohomology as given by the toric Serre-Grothendieck correspondence (\cite[4.5]{quasi}).)


\section{Equality of torsion functors: The general case}

In this section we search for conditions under which torsion functors over arbitrary rings depend only on the radical of the supporting ideals. For the sake of generality we consider throughout a graded setting, even if the graduation is not relevant. So, let $G$ be a commutative group and let $R$ be a $G$-graded ring.

By a projective system of graded ideals we mean a projective system of graded ideals of $R$ over an ordered set. We denote such a system in the form $\mathfrak{A}=(\mathfrak{a}_j)_{j\in J}$, where $J$ is understood to be an ordered set. For a projective system of graded ideals $\mathfrak{A}=(\mathfrak{a}_j)_{j\in J}$ we denote by $\Gamma_{\mathfrak{A}}$ the $G$-graded $\mathfrak{A}$-torsion functor, i.e., the subfunctor of the identity functor on the category ${\sf GrMod}^G(R)$ of $G$-graded $R$-modules with $\Gamma_{\mathfrak{A}}(M)=\bigcup_{j\in J}(0:_M\mathfrak{a}_j)$ for a $G$-graded $R$-module $M$. Extending our ambient universe (cf.~\cite[I.0]{sga4}) we can consider the (potentially big) set $$\mathbbm{T}(R)\dfgl\{\Gamma_{\mathfrak{A}}\mid\mathfrak{A}\text{ projective system of graded ideals}\}$$ of subfunctors of the identity functor as an ordered set by means of its canonical ordering. If $\mathfrak{A}=(\mathfrak{a}_j)_{j\in J}$ and $\mathfrak{B}=(\mathfrak{b}_k)_{k\in K}$ are two projective systems of graded ideals, then $\mathfrak{A}$ is called \textit{coarser} than $\mathfrak{B}$ if for every $j\in J$ there exists $k\in K$ with $\mathfrak{b}_k\subseteq\mathfrak{a}_j$; this defines a preorder relation. Two projective systems of graded ideals are called \textit{equivalent} if each of them is coarser than the other; this defines an equivalence relation. Slightly generalising \cite[III.1.2]{e} (and keeping in mind \cite[II.6.9]{e}) we get representatives of equivalence classes under this relation, and thus obtain an order relation associated with the above preorder relation. This order relation is readily seen to have a graph, and so we end up with a (small) ordered set $\mathbbm{S}(R)$ whose elements are representatives of equivalence classes of projective systems of graded ideals. The relation between $\mathbbm{S}(R)$ and $\mathbbm{T}(R)$ is given by the following proposition, where by an isomorphism of ordered sets we understand -- accordingly to \cite[III.1.3]{e} -- an increasing bijection with increasing inverse.

\begin{prop}
There is an isomorphism of ordered sets $$\mathbbm{S}(R)\overset{\cong}\longrightarrow\mathbbm{T}(R),\;\mathfrak{A}\mapsto\Gamma_{\mathfrak{A}}.$$
\end{prop}

\begin{proof}
The map in question is obviously a surjective morphism of ordered sets. Let $\mathfrak{A}=(\mathfrak{a}_j)_{j\in J}$ and $\mathfrak{B}=(\mathfrak{b}_k)_{k\in K}$ be projective systems of graded ideals such that $\Gamma_{\mathfrak{A}}$ is a subfunctor of $\Gamma_{\mathfrak{B}}$. If $j\in J$ then $R/\mathfrak{a}_j$ is annihilated by $\mathfrak{a}_j$, and so $R/\mathfrak{a}_j=\Gamma_{\mathfrak{A}}(R/\mathfrak{a}_j)\subseteq\Gamma_{\mathfrak{B}}(R/\mathfrak{a}_j)$. Hence, there exists $k\in K$ with $\mathfrak{b}_k\subseteq\mathfrak{a}_j$, and thus $\mathfrak{A}$ is coarser than $\mathfrak{B}$. The claim follows from this.
\end{proof}

In particular, it follows that the set $\mathbbm{T}(R)$ is small. For a graded ideal $\mathfrak{a}\subseteq R$ we consider the projective system of graded ideals $\mathfrak{A}=(\mathfrak{a}^n)_{n\in\mathbbm{N}}$ and set $\Gamma_{\mathfrak{a}}\dfgl\Gamma_{\mathfrak{A}}$, and we denote by $\mathfrak{T}_{\mathfrak{a}}$ the $\mathfrak{a}$-adic topology on $R$. If $\mathfrak{a},\mathfrak{b}\subseteq R$ are two graded ideals, then $(\mathfrak{a}^n)_{n\in\mathbbm{N}}$ is coarser than $(\mathfrak{b}^n)_{n\in\mathbbm{N}}$ if and only if $\mathfrak{T}_{\mathfrak{a}}$ is coarser than $\mathfrak{T}_{\mathfrak{b}}$ (\cite[III.2.5]{ac}).

\begin{cor}\label{1.20}
There is an isomorphism of ordered sets $$\{\mathfrak{T}_{\mathfrak{a}}\mid\mathfrak{a}\subseteq R\text{ graded ideal}\,\}\overset{\cong}\longrightarrow\{\Gamma_{\mathfrak{a}}\mid\mathfrak{a}\subseteq R\text{ graded ideal}\,\}.$$
\end{cor}

\begin{cor}\label{1.30}
a) If\/ $\mathfrak{a}\subseteq R$ is a graded ideal, then $\Gamma_{\mathfrak{a}}=\Gamma_{\sqrt{\mathfrak{a}}}$ holds if and only if there exists $n\in\mathbbm{N}$ with $\sqrt{\mathfrak{a}}^n\subseteq\mathfrak{a}$.

b) If\/ $\mathfrak{a},\mathfrak{b}\subseteq R$ are graded ideals, then $\Gamma_{\mathfrak{a}}=\Gamma_{\mathfrak{b}}$ implies $\sqrt{\mathfrak{a}}=\sqrt{\mathfrak{b}}$, and the converse holds if there exist $m,n\in\mathbbm{N}$ with $\sqrt{\mathfrak{a}}^m\subseteq\mathfrak{a}$ and $\sqrt{\mathfrak{b}}^n\subseteq\mathfrak{b}$.
\end{cor}

If $\mathfrak{a}\subseteq R$ is a graded ideal whose radical is of finite type then $\Gamma_{\mathfrak{a}}=\Gamma_{\sqrt{\mathfrak{a}}}$ holds by \ref{1.30}~a). Hence, if $R$ is Noetherian (where Noetherianness of a graded ring is always understood to mean that every graded ideal is of finite type) then we have $\Gamma_{\mathfrak{a}}=\Gamma_{\sqrt{\mathfrak{a}}}$ for every graded ideal $\mathfrak{a}\subseteq R$. Similarly, the hypothesis for the converse to hold in \ref{1.30}~b) is fulfilled if the radicals of $\mathfrak{a}$ and $\mathfrak{b}$ are of finite type, hence in particular if $R$ is Noetherian. The seemingly weaker hypothesis that every perfect graded ideal of $R$ is of finite type yields no improvement -- it is equivalent to Noetherianness of $R$ by a graded version of Cohen's Theorem \cite[0.6.4.7]{ega}, proved analogously to the ungraded one.

The above shows that even when considering only graded ideals of finite type we do not have a satisfying result. Therefore, we will restrict our attention to a more specific situation in the next section.


\section{Equality of torsion functors: The monomial case}

In this section we consider the following situation. Let $A$ be a ring, let $I$ be a set, and let $\psi\colon\mathbbm{Z}^{\oplus I}\twoheadrightarrow G$ be an epimorphism of groups. We denote by $R$ the polynomial algebra $A[(X_i)_{i\in I}]$ in indeterminates $(X_i)_{i\in I}$ over $A$, furnished with the $G$-graduation derived from its canonical $\mathbbm{Z}^{\oplus I}$-graduation by means of $\psi$. By a \textit{monomial (in $R$)} we mean an element of $R$ of the form $\prod_{i\in I}X_i^{\mu_i}$ with $\mu=(\mu_i)_{i\in I}\in\mathbbm{N}^{\oplus I}$, and such an element is homogeneous of degree $\psi(\mu)$. By a \textit{monomial ideal (of $R$)} we mean an ideal of $R$ generated by a set of monomials in $R$. Monomial ideals are graded, but the converse is not necessarily true. For a subset $E\subseteq R$ we write $\langle E\rangle_R$ or, if no confusion can arise, $\langle E\rangle$ for the ideal of $R$ generated by $E$.

The ordered set $\mathbbm{N}^{\oplus I}$ (whose ordering is induced by the product ordering on $\mathbbm{N}^I$) is Artinian. Therefore, for a monomial ideal $\mathfrak{a}\subseteq R$ there exists a unique subset $E\subseteq\mathbbm{N}^{\oplus I}$ such that $\{\prod_{i\in I}X_i^{\mu_i}\mid\mu\in E\}$ is a minimal generating set of $\mathfrak{a}$; we denote this set by $E(\mathfrak{a})$ and define a monomial ideal $$\textstyle\lfloor\mathfrak{a}\rfloor\dfgl\langle\prod_{i\in I,\mu_i>0}X_i\mid\mu\in E(\mathfrak{a})\rangle_R.$$ This contains but not necessarily equals $\mathfrak{a}$.

\begin{lemma}\label{2.04}
If\/ $\mathfrak{a}\subseteq R$ is a monomial ideal, then $\mathfrak{a}$ is of finite type if and only if $\lfloor\mathfrak{a}\rfloor$ is so.
\end{lemma}

\begin{proof}
If $\mathfrak{a}$ is of finite type then $E(\mathfrak{a})$ is finite, hence $\lfloor\mathfrak{a}\rfloor$ is of finite type, too. Conversely, suppose that $\lfloor\mathfrak{a}\rfloor$ is of finite type, and assume that $\mathfrak{a}$ is not of finite type. Then, by the pigeonhole principle there exists an infinite subset $F\subseteq E(\mathfrak{a})$ such that for $\mu,\nu\in F$ we have $\prod_{i\in I,\mu_i>0}X_i=\prod_{i\in I,\nu_i>0}X_i$. Thus, there exists a finite subset $J\subseteq I$ with $F\subseteq\mathbbm{N}^{\oplus J}$. But then $F=E(\langle\prod_{i\in J}X_i^{\mu_i}\mid\mu\in F\rangle)$ is the minimal set of generators of a monomial ideal in a polynomial ring in finitely many indeterminates, hence finite -- a contradiction.
\end{proof}

For a ring $B$ we denote by $\nil(B)$ its nilradical. If $\nil(B)$ is finitely generated then it is nilpotent, but the converse is not necessarily true. In our situation we have $\nil(R)=\nil(A)R$, i.e., a polynomial in $R$ is nilpotent if and only if its coefficients are nilpotent (\cite[IV.1.5 Proposition 9]{a}). Hence, $\nil(R)$ is nilpotent or of finite type, respectively, if and only if $\nil(A)$ is so.

\begin{lemma}\label{2.05}
a) For a monomial ideal $\mathfrak{a}\subseteq R$ we have $\sqrt{\mathfrak{a}}=\nil(R)+\lfloor\mathfrak{a}\rfloor$.

b) For monomial ideals $\mathfrak{a},\mathfrak{b}\subseteq R$ we have $\sqrt{\mathfrak{a}}=\sqrt{\mathfrak{b}}$ if and only if\/ $\lfloor\mathfrak{a}\rfloor=\lfloor\mathfrak{b}\rfloor$.
\end{lemma}

\begin{proof}
a) Clearly, $\nil(R)+\lfloor\mathfrak{a}\rfloor$ is contained in $\sqrt{\mathfrak{a}}$. Conversely, let $f\in\sqrt{\mathfrak{a}}$. Every term occurring in $f$ being either nilpotent or not, and the sum of the nilpotent terms being nilpotent, we can write $f=g+h$ with $g\in\nil(R)$ and $h\in R$ such that no term occurring in $h$ is nilpotent. So, it suffices to show $h\in\lfloor\mathfrak{a}\rfloor$. As $f,g\in\sqrt{\mathfrak{a}}$ it follows $h\in\sqrt{\mathfrak{a}}$, and hence we can suppose that no term occurring in $f$ is nilpotent. We prove now by induction on the number $r$ of terms occurring in $f$ that $f\in\lfloor\mathfrak{a}\rfloor$. If $r=0$ this is clear. Suppose that $r>0$ and that the claim holds for strictly smaller values of $r$. Furnishing $\mathbbm{Z}^{\oplus I}$ with a structure of totally ordered group (e.g., the lexicographic product with respect to some well-ordering of $I$) there occurs a term $t$ in $f$ whose $\mathbbm{Z}^{\oplus I}$-degree is the greatest of the $\mathbbm{Z}^{\oplus I}$-degrees of the terms occurring in $f$. If $n\in\mathbbm{N}$ with $f^n\in\mathfrak{a}$ then $t^n$ is a term occurring in $f^n$ whose $\mathbbm{Z}^{\oplus I}$-degree is not shared by any other term occurring in $f^n$. Hence, as $\mathfrak{a}$ is a monomial ideal, $\prod_{i\in I}X_i^{\mu_i}$ divides $t^n$ for some $\mu\in E(\mathfrak{a})$, and thus $t\in\lfloor\mathfrak{a}\rfloor\subseteq\sqrt{\mathfrak{a}}$. We get $f'\dfgl f-t\in\sqrt{\mathfrak{a}}$. As the number of terms occurring in $f'$ is strictly smaller than $r$ it follows $f'\in\lfloor\mathfrak{a}\rfloor$, hence $f=f'+t\in\lfloor\mathfrak{a}\rfloor$ as desired.

b) Suppose that $\sqrt{\mathfrak{a}}=\sqrt{\mathfrak{b}}$ and therefore $\lfloor\mathfrak{a}\rfloor\subseteq\nil(R)+\lfloor\mathfrak{b}\rfloor$ by a). For $\mu\in E(\mathfrak{a})$ it follows $\prod_{i\in I,\mu_i>0}X_i\in\lfloor\mathfrak{a}\rfloor\subseteq\nil(R)+\lfloor\mathfrak{b}\rfloor$, hence either $\prod_{i\in I,\mu_i>0}X_i\in\nil(R)$ and thus $R=0$, or $\prod_{i\in I,\mu_i>0}X_i\in\lfloor\mathfrak{b}\rfloor$. So, by reasons of symmetry we get $\lfloor\mathfrak{a}\rfloor=\lfloor\mathfrak{b}\rfloor$. The converse follows from a).
\end{proof}

\begin{prop}\label{2.10}
a) For a monomial ideal $\mathfrak{a}\subseteq R$ of finite type we have $\Gamma_{\mathfrak{a}}=\Gamma_{\lfloor\mathfrak{a}\rfloor}$.

b) If $\mathfrak{a},\mathfrak{b}\subseteq R$ are monomial ideals of finite type, then the following statements are equivalent: (i) $\Gamma_{\mathfrak{a}}=\Gamma_{\mathfrak{b}}$; (ii) $\lfloor\mathfrak{a}\rfloor=\lfloor\mathfrak{b}\rfloor$; (iii) $\sqrt{\mathfrak{a}}=\sqrt{\mathfrak{b}}$.
\end{prop}

\begin{proof}
a) For $\mu\in E(\mathfrak{a})$ and $m\in\mathbbm{N}$ with $m\geq\mu_i$ for every $i\in I$ we have $(\prod_{i\in I,\mu_i>0}X_i)^m\in\mathfrak{a}$. As $\lfloor\mathfrak{a}\rfloor$ is of finite type by \ref{2.04} this implies that there exists $m\in\mathbbm{N}$ with $\lfloor\mathfrak{a}\rfloor^m\subseteq\mathfrak{a}$. Hence, as $\mathfrak{a}\subseteq\lfloor\mathfrak{a}\rfloor$ the claim follows from \ref{1.20}.

b) (i) implies (iii) by \ref{1.30}~b), (iii) implies (ii) by \ref{2.05}~b), and (ii) implies (i) by a).
\end{proof}

This answers our original question in a satisfying way. But since the radical of a monomial ideal of finite type is not necessarily a monomial ideal we cannot infer that $\Gamma_{\mathfrak{a}}=\Gamma_{\sqrt{\mathfrak{a}}}$. Moreover, in case of monomial ideals that are not of finite type \ref{2.10}~b) is no longer true. Let us provide some examples of good and bad behaviour.

\begin{exas}
Let $I=\mathbbm{N}_{>0}$ and $\psi={\rm Id}_{\mathbbm{Z}^{\oplus I}}$, let $\mathfrak{a}\subseteq R$ be a monomial ideal of finite type, and let $\mathfrak{b}\dfgl\langle X_i^i\mid i\in I\rangle_R$. If $A=\mathbbm{Q}[(Y_k)_{k\in\mathbbm{N}}]/\langle Y_k^2\mid k\in I\rangle$ then we have $\Gamma_{\mathfrak{a}}=\Gamma_{\lfloor\mathfrak{a}\rfloor}=\Gamma_{\sqrt{\mathfrak{a}}}$ and $\Gamma_{\mathfrak{b}}\neq\Gamma_{\lfloor\mathfrak{b}\rfloor}=\Gamma_{\sqrt{\mathfrak{b}}}$. If $A=\mathbbm{Q}[(Y_k)_{k\in\mathbbm{N}}]/\langle Y_k^k\mid k\in I\rangle$ then we have $\Gamma_{\mathfrak{a}}=\Gamma_{\lfloor\mathfrak{a}\rfloor}\neq\Gamma_{\sqrt{\mathfrak{a}}}$ and $\Gamma_{\mathfrak{b}}\neq\Gamma_{\lfloor\mathfrak{b}\rfloor}\neq\Gamma_{\sqrt{\mathfrak{b}}}\neq\Gamma_{\mathfrak{b}}$. 
\end{exas}

These examples already hint at the condition on the base ring for $\Gamma_{\lfloor\mathfrak{a}\rfloor}$ and $\Gamma_{\sqrt{\mathfrak{a}}}$ (and in case $\mathfrak{a}$ is of finite type $\Gamma_{\mathfrak{a}}$ and $\Gamma_{\sqrt{\mathfrak{a}}}$) to coincide, as presented in the next result.

\begin{prop}
The following statements are equivalent: (i) $\nil(A)$ is nilpotent; (ii) For every monomial ideal $\mathfrak{a}\subseteq R$ we have $\Gamma_{\lfloor\mathfrak{a}\rfloor}=\Gamma_{\sqrt{\mathfrak{a}}}$; (iii) There exists a monomial ideal $\mathfrak{a}\subseteq R$ with $\Gamma_{\lfloor\mathfrak{a}\rfloor}=\Gamma_{\sqrt{\mathfrak{a}}}$.
\end{prop}

\begin{proof}
``(i)$\Rightarrow$(ii)'': Immediately by \ref{1.20} and \ref{2.05}~a). ``(iii)$\Rightarrow$(i)'': There exists $n\in\mathbbm{N}$ with $\sqrt{\mathfrak{a}}^n\subseteq\lfloor\mathfrak{a}\rfloor$ (\ref{1.20}), implying $\nil(A)^n\subseteq\nil(R)^n\cap A\subseteq\lfloor\mathfrak{a}\rfloor\cap A=0$ (\ref{2.05}~a)).
\end{proof}

Sometimes -- for example in toric geometry -- it is necessary to consider a slightly more general situation than the one above, namely monomial ideals in subrings of $R$ obtained by degree restriction to a subgroup of finite index of $G$. So, from now on let $H\subseteq G$ be a subgroup of finite index. We consider the degree restriction $S\dfgl R_{(H)}$ of $R$ to $H$, i.e., the sub-$A$-algebra $A[\prod_{i\in I}X_i^{\mu_i}\mid\mu\in\mathbbm{N}^{\oplus I}\wedge\psi(\mu)\in H]$ of $R$ generated by the monomials with degree in $H$, furnished with the $H$-graduation such that $\deg(\prod_{i\in I}X_i^{\mu_i})=\psi(\mu)$. By a \textit{monomial in $S$} we mean a monomial in $R$ contained in $S$, and by a \textit{monomial ideal of $S$} we mean an ideal of $S$ generated by monomials in $S$. Monomial ideals of $S$ are graded, but the converse is not necessarily true. For a monomial ideal $\mathfrak{a}$ of $R$ we consider the degree restriction $\mathfrak{a}_{(H)}\dfgl\mathfrak{a}\cap S$ to $H$. This is a monomial ideal of $S$, and if $\mathfrak{a}$ is of finite type then so is $\mathfrak{a}_{(H)}$ (as is seen analogously to \cite[3.1.3 a)]{quasi}). Conversely, if $\mathfrak{a}$ is a monomial ideal (of finite type) of $S$ then $\langle\mathfrak{a}\rangle_R$ is a monomial ideal (of finite type) of $R$, and we have $(\langle\mathfrak{a}\rangle_R)_{(H)}=\mathfrak{a}$ (\cite[2.2.3]{quasi}). Therefore, every monomial ideal (of finite type) of $S$ is of the form $\mathfrak{a}_{(H)}$ for some monomial ideal $\mathfrak{a}$ (of finite type) of $R$.

In this situation there are two ways of forming a ``monomial radical'' of a monomial ideal. We will show now that both yield the same torsion functors, but only one of them can be used to characterise equality of torsion functors. So, let $\mathfrak{a}\subseteq R$ be a monomial ideal. Since $H$ has finite index in $G$ it follows that for $\mu\in E(\mathfrak{a})$ there exists a minimal $m\in\mathbbm{N}$ with the property that $\prod_{i\in I,\mu_i>0}X_i^m\in\mathfrak{a}_{(H)}$, and we denote this number by $m_{\mu}$. Then, $$\textstyle\lfloor\mathfrak{a}\rfloor_H\dfgl\langle\prod_{i\in I,\mu_i>0}X_i^{m_{\mu}}\mid\mu\in E(\mathfrak{a})\rangle_S\subseteq S$$ is a monomial ideal. On the other hand, we can also consider the monomial ideal $\lfloor\mathfrak{a}\rfloor_{(H)}\subseteq S$, obtained by degree restriction from $\lfloor\mathfrak{a}\rfloor$. Clearly, both these constructions generalise the one above, corresponding to the case $H=G$.

It is readily checked that $\mathfrak{a}_{(H)}\subseteq\lfloor\mathfrak{a}\rfloor_H\subseteq\lfloor\mathfrak{a}\rfloor_{(H)}$, but none of these inclusions is necessarily an equality. For the first this is clear by considering the case $H=G$. For the second we suppose that $G=\mathbbm{Z}$, $\deg(X_0)=1$, and $H=2\mathbbm{Z}$. Then, setting $\mathfrak{a}=\langle X_0^4\rangle_R$ we get $\mathfrak{a}_{(H)}=\lfloor\mathfrak{a}\rfloor_H=\langle X_0^4\rangle_S\subsetneqq\langle X_0^2\rangle_S=\langle\mathfrak{a}\rangle_{(H)}$ (provided $A\neq 0$).

\begin{lemma}
If\/ $\mathfrak{a}\subseteq R$ is a monomial ideal, then $\mathfrak{a}$ is of finite type if and only if $\lfloor\mathfrak{a}\rfloor_H$ is so.
\end{lemma}

\begin{proof}
Analogously to \ref{2.04}.
\end{proof}

\begin{prop}\label{2.30}
For a monomial ideal $\mathfrak{a}\subseteq R$ of finite type we have $$\Gamma_{\mathfrak{a}_{(H)}}=\Gamma_{\lfloor\mathfrak{a}\rfloor_H}=\Gamma_{\lfloor\mathfrak{a}\rfloor_{(H)}}.$$
\end{prop}

\begin{proof}
Let $t\in\lfloor\mathfrak{a}\rfloor_{(H)}$ be a monomial. Then, there exist $\mu\in E(\mathfrak{a})$ and $\nu\in\mathbbm{N}^{\oplus I}$ with $t=\prod_{i\in I,\mu_i>0}X_i^{1+\nu_i}\cdot\prod_{i\in I,\mu_i=0}X_i^{\nu_i}$, and there exists $n\in\mathbbm{N}$ such that for $i\in I$ with $\mu_i>0$ we have $n(1+\nu_i)\geq\mu_i$. This implies that $t^n$ is a multiple of $\prod_{i\in I}X_i^{\mu_i}$ and therefore belongs to $\mathfrak{a}$. Since $\lfloor\mathfrak{a}\rfloor_{(H)}\subseteq S$ is by \ref{2.04} a monomial ideal of finite type, the claim follows from \ref{1.20}.
\end{proof}

\begin{prop}\label{2.40}
Let $\mathfrak{a},\mathfrak{b}\subseteq R$ be monomial ideals of finite type.

a) The following statements are equivalent: (i) $\Gamma_{\mathfrak{a}_{(H)}}=\Gamma_{\mathfrak{b}_{(H)}}$; (ii) $\sqrt{\mathfrak{a}}=\sqrt{\mathfrak{b}}$; (iii) $\lfloor\mathfrak{a}\rfloor=\lfloor\mathfrak{b}\rfloor$; (iv) $\lfloor\mathfrak{a}\rfloor_{(H)}=\lfloor\mathfrak{b}\rfloor_{(H)}$.

b) If\/ $\lfloor\mathfrak{a}\rfloor_H=\lfloor\mathfrak{b}\rfloor_H$ then $\Gamma_{\mathfrak{a}_{(H)}}=\Gamma_{\mathfrak{b}_{(H)}}$.
\end{prop}

\begin{proof}
a) ``(i)$\Rightarrow$(iii)'': If $\Gamma_{\mathfrak{a}_{(H)}}=\Gamma_{\mathfrak{b}_{(H)}}$ then $\sqrt{\lfloor\mathfrak{a}\rfloor_H}=\sqrt{\lfloor\mathfrak{b}\rfloor_H}$ holds (\ref{2.30}, \ref{1.30}). Let $\mu\in E(\mathfrak{a})$. There exists $n\in\mathbbm{N}_{>0}$ with $\prod_{i\in I,\mu_i>0}X_i^{nm_{\mu}}\in\lfloor\mathfrak{b}\rfloor_H$, so there exists $\nu\in E(\mathfrak{b})$ such that $\prod_{i\in I,\nu_i>0}X_i^{m_{\nu}}$ divides $\prod_{i\in I,\mu_i>0}X_i^{nm_{\mu}}$. If $m_{\nu}=0$ or $m_{\mu}=0$ then we get $\mathfrak{a}=R=\mathfrak{b}$ and thus the claim. Otherwise, $\prod_{i\in I,\nu_i>0}X_i$ divides $\prod_{i\in I,\mu_i>0}X_i$, implying $\prod_{i\in I,\mu_i>0}X_i\in\lfloor\mathfrak{b}\rfloor$ and thus by symmetry the claim. The remaining implications follow trivially or by \ref{2.10}~b) and \ref{2.30}.

b) The hypothesis implies $\sqrt{\lfloor\mathfrak{a}\rfloor_H}=\sqrt{\lfloor\mathfrak{b}\rfloor_H}$, and the claim follows with the same argument as the implication ``(i)$\Rightarrow$(iii)'' in a).
\end{proof}

The converse of \ref{2.40}~b) is not true. Indeed, supposing again that $G=\mathbbm{Z}$, $\deg(X_0)=1$, and $H=2\mathbbm{Z}$, and setting $\mathfrak{a}=\langle X_0^4\rangle_R$ and $\mathfrak{b}=\langle X_0^2\rangle_R$, we get $\lfloor\mathfrak{a}\rfloor=\langle X_0\rangle_R=\lfloor\mathfrak{b}\rfloor$, but $\lfloor\mathfrak{a}\rfloor_H=\langle X_0^4\rangle_S\subsetneqq\langle X_0^2\rangle_S=\lfloor\mathfrak{b}\rfloor_H$ (provided $A\neq 0$).


\section{Application: Flatness of quasicoherent sheaves on Cox schemes}

In this last section we will apply our results to toric geometry. For the necessary background and unexplained notations we refer the reader to \cite{quasi}.

Before recalling the toric setting we have to make some general remarks on flatness and \v{C}ech cohomology. During these, let $G$ be a commutative group, let $A$ be a ring, and let $R$ be a $G$-graded ring such that $R_0$ is an $A$-algebra. For a $G$-graded $R$-module $M$ we set $\degsupp(M)\dfgl\{g\in G\mid M_g\neq 0\}$. For a finite sequence ${\bf a}=(a_j)_{j=1}^n$ in $R^{\hom}$ we denote by $C({\bf a},M)$ the $G$-graded \v{C}ech cocomplex of $M$ with respect to ${\bf a}$, by $C({\bf a},M)^i$ its $i$-th component for $i\in\mathbbm{Z}$, and by $H^i({\bf a},M)$ the $i$-th $G$-graded \v{C}ech cohomology of $M$ with respect to ${\bf a}$ for $i\in\mathbbm{Z}$. Furthermore, for a graded ideal $\mathfrak{a}\subseteq R$ and $i\in\mathbbm{Z}$ we denote by $H_{\mathfrak{a}}^i(M)$ the $i$-th $G$-graded local cohomology of $M$ with respect to $\mathfrak{a}$. One should note that if $R$ is not Noetherian then \v{C}ech cohomology with respect to ${\bf a}$ and local cohomology with respect to the ideal of $R$ generated by ${\bf a}$ do not necessarily coincide.

\begin{lemma}\label{3.20}
Let $R'$ be a flat $G$-graded $R$-algebra, and let $M$ be a $G$-graded $R$-module. If $M$ is flat over $A$, then so is $R'\otimes_RM$.
\end{lemma}

\begin{proof}
The functor $(R'\otimes_RM)\otimes_A\bullet\colon{\sf Mod}(A)\rightarrow{\sf Mod}(A)$ is the composition of $$R'\otimes_R(M\otimes_A\bullet)\colon{\sf Mod}(A)\rightarrow{\sf GrMod}^G(R')$$ with the forgetful functor from ${\sf GrMod}^G(R')$ to ${\sf Mod}(R')$ and the scalar restriction from ${\sf Mod}(R')$ to ${\sf Mod}(A)$, and all these three functors are exact. 
\end{proof}

\begin{lemma}\label{3.30}
Let ${\bf a}=(a_j)_{j=1}^n$ be a finite sequence in $R^{\hom}$, and let $M$ be a $G$-graded $R$-module. If $M_{a_j}$ is flat over $A$ for every $j\in[1,n]$, then so is $C({\bf a},M)^i_g$ for every $i\in\mathbbm{N}_{>0}$ and every $g\in G$.
\end{lemma}

\begin{proof}
Using \ref{3.20} and \cite[I.2.3 Proposition 2]{ac} this follows immediately from the definition of the \v{C}ech cocomplex.
\end{proof}

\begin{prop}\label{3.40}
Let ${\bf a}=(a_j)_{j=1}^n$ be a finite sequence in $R^{\hom}$, and let $M$ be a $G$-graded $R$-module. If $M_{a_j}$ is flat over $A$ for every $j\in[1,n]$, then so is $M_g$ for every $g\in G\setminus\bigcup_{i=0}^n\degsupp(H^i({\bf a},M))$.
\end{prop}

\begin{proof}
Taking components of degree $g\in G\setminus\bigcup_{i=0}^n\degsupp(H^i({\bf a},M))$ of $C({\bf a},M)$ yields an exact sequence $$0\rightarrow C({\bf a},M)^0_g\rightarrow C({\bf a},M)^1_g\rightarrow\cdots\rightarrow C({\bf a},M)^{n-1}_g\rightarrow C({\bf a},M)^n_g\rightarrow 0$$ of $A$-modules. Using \ref{3.30} and \cite[I.2.5 Proposition 5]{ac} the claim follows then by induction on $n$.
\end{proof}

Now we are ready to consider the toric setting. Let $V$ be an $\mathbbm{R}$-vector space of finite dimension, let $N$ be a $\mathbbm{Z}$-structure on $V$, let $\Sigma$ be an $N$-fan in $V$, let $\Sigma_1$ denote the set of $1$-dimensional cones in $\Sigma$, and let $R$ be a ring. The combinatorics of $\Sigma$ give rise to a group $A$ and an epimorphism $a\colon\mathbbm{Z}^{\Sigma_1}\twoheadrightarrow A$, as well as to a subgroup $\pic(\Sigma)\subseteq A$, the so-called Picard group of $\Sigma$. Additionally, we choose a subgroup $B\subseteq A$ of finite index. The Cox ring $S$ associated with $\Sigma_1$ over $R$ is defined as the polynomial algebra $R[(Z_{\rho})_{\rho\in\Sigma_1}]$ furnished with the $A$-graduation derived from its canonical $\mathbbm{Z}^{\Sigma_1}$-graduation by means of $a$, and the $B$-restricted Cox ring $S_B$ is defined as its degree restriction to $B$. For $\sigma\in\Sigma$ we define a monomial $\widehat{Z}_{\sigma}\dfgl\prod_{\rho\in\Sigma_1\setminus\sigma}Z_{\rho}\in S$. For $\sigma\in\Sigma$ there exists a minimal $m\in\mathbbm{N}$ with the property that $\widehat{Z}_{\sigma}^m\in S_B$, and we denote this number by $m_{\sigma}$. We denote the ring of fractions obtained from $S_B$ by inverting $\widehat{Z}_{\sigma}^{m_{\sigma}}$ by $S_{B,\sigma}$. Its component of degree $0$ is independent of the choice of $B$ and is denoted by $S_{(\sigma)}$. The facial relation in $\Sigma$ allows us to glue the $R$-schemes $\spec(S_{(\sigma)})$ to obtain an $R$-scheme $Y\dfgl Y_{\Sigma}(R)$, called the Cox scheme associated with $\Sigma$ over $R$. There is a canonical morphism of $R$-schemes from $Y$ to the toric scheme associated with $\Sigma$ over $R$, which is an isomorphism if and only if $\Sigma$ is full, i.e., $\langle\bigcup\Sigma\rangle_{\mathbbm{R}}=V$. Hence, studying Cox schemes is essentially the same as studying toric schemes.

There is an essentially surjective functor $\mathscr{S}_B$ from the category of $B$-graded $S_B$-modules to the category of quasicoherent $\mathscr{O}_Y$-modules, allowing to translate the study of quasicoherent sheaves on Cox schemes (and hence on toric schemes) into the study of $B$-graded $S_B$-modules. For example, in \cite[4.2.7]{quasi} it was shown that if $F$ is a $B$-graded $S_B$-module such that $S_{B,\sigma}\otimes_{S_B}F$ is flat over $R$ for every $\sigma\in\Sigma$, then $\mathscr{S}_B(F)$ is flat over $R$, and the converse holds if $B\subseteq\pic(\Sigma)$ (in which case $\Sigma$ is simplicial by \cite[1.4.5]{quasi}). Thus, after choosing a counting $\widehat{Z}$ of $(\widehat{Z}_{\sigma}^{m_{\sigma}})_{\sigma\in\Sigma}$ we can relate flatness of $\mathscr{S}_B(F)$ with \v{C}ech cohomology of $F$.

\begin{prop}\label{3.50}
Suppose that $B\subseteq\pic(\Sigma)$, and let $F$ be a $B$-graded $S_B$-module. If $\mathscr{S}_B(F)$ is flat over $R$, then $F_{\alpha}$ is flat over $R$ for every $\alpha\in B\setminus\bigcup_{i\in\mathbbm{N}}\degsupp(H^i(\widehat{Z},F))$.
\end{prop}

\begin{proof}
Immediately from the above and \ref{3.40}.
\end{proof}

The monomial ideal $I\dfgl\langle\widehat{Z}_{\sigma}\mid\sigma\in\Sigma\rangle_S$ of $S$ is called the irrelevant ideal, and its degree restriction $I_B\subseteq S_B$ to $B$ is a monomial ideal of finite type. Under the condition that $S_B$ has the ITI-property with respect to $I_B$, i.e., that $\Gamma_{I_B}$ preserves injectivity of $B$-graded $S_B$-modules, the toric Serre-Grothendieck correspondence (\cite[4.5.4]{quasi}) states that we can study sheaf cohomology on $Y$ in terms of $B$-graded local cohomology over $S_B$ with support in $I_B$. Thus, local cohomology with respect to $I_B$ seems to be an appropriate tool in our setting, and the question arises on its relation with \v{C}ech cohomology with respect to $\widehat{Z}$. Under the condition that $S_B$ has the ITI-property with respect to $\langle\widehat{Z}_{\sigma}^{m_{\sigma}}\rangle_{S_B}$ for every $\sigma\in\Sigma$, our results so far provide a satisfying answer.

\begin{cor}\label{3.60}
Suppose that $B\subseteq\pic(\Sigma)$ and that $S_B$ has the ITI-property with respect to $\langle\widehat{Z}_{\sigma}^{m_{\sigma}}\rangle_{S_B}$ for every $\sigma\in\Sigma$, and let $F$ be a $B$-graded $S_B$-module. If $\mathscr{S}_B(F)$ is flat over $R$, then $F_{\alpha}$ is flat over $R$ for every $\alpha\in B\setminus\bigcup_{i\in\mathbbm{N}}\degsupp(H^i_{I_B}(F))$.
\end{cor}

\begin{proof}
Since $\langle\widehat{Z}_{\sigma}^{m_{\sigma}}\mid\sigma\in\Sigma\rangle_{S_B}=\lfloor I\rfloor_B$ this follows from \ref{2.30}, \ref{3.50} and the proof of \cite[Proposition 7]{coarsening}.
\end{proof}

It remains the question about a converse of the above result, i.e., Does flatness of certain components of $F$ imply flatness of $\mathscr{S}_B(F)$? Considering the analogous results in the case of projective schemes one should be able to tackle this problem using bounds for some notion of Castelnuovo-Mumford regularity appropriate to our setting (cf.~\cite{bc}, \cite{ms}). We hope to address this in some future work.


\medskip

\noindent\textbf{Acknowledgement:} I thank the referee for the very careful reading and for his comments and suggestions.


\end{document}